\documentclass[11pt,reqno]{amsart}

\usepackage{amsfonts, amsthm, amsmath}
\allowdisplaybreaks[4]

\usepackage{rotating}

\usepackage{tikz}

\usetikzlibrary{decorations.markings} 

\usepackage{graphics}

\usepackage{amssymb}

\usepackage{amscd}

\usepackage[latin2]{inputenc}

\usepackage{t1enc}

\usepackage[mathscr]{eucal}

\usepackage{indentfirst}

\usepackage{graphicx}

\usepackage{graphics}

\usepackage{pict2e}

\usepackage{mathrsfs}

\usepackage{enumerate}
\usepackage{url}
\usepackage[pagebackref]{hyperref}
\hypersetup{colorlinks=true}
\usepackage{cite}
\usepackage{color}
\usepackage{epic}
\usepackage{hyperref} 
\usepackage{framed}
\usepackage{mathabx}
\usepackage{scalerel}
\usepackage{booktabs}
\usepackage{makecell}
\newcolumntype{V}{!{\vrule width 2pt}}

\numberwithin{equation}{section}
\theoremstyle{plain}

\newtheorem{theorem}{Theorem}[section]

\newtheorem{lemma}[theorem]{Lemma}

\theoremstyle{definition}

\newtheorem{Def}[theorem]{Definition}

\newtheorem{example}[theorem]{Example}

\newtheorem{remark}{Remark}

\newtheorem{?}[theorem]{Problem}

\newcommand{\R}{\mathbb{R}}

\begin{document}

\title{Polynomization of the Liu-Zhang inequality for overpartition function }
\author[X. Li]{Xixi Li}
\address[Xixi Li]{College of Mathematics and Statistics, Chongqing University, Huxi campus, Chongqing 401331, P.R. China}
\email{2472005992@qq.com}
\date{\today}

\begin{abstract}
Let $\overline{p}(n)$ denote the overpartition function. Liu and Zhang showed that $\overline{p}(a) \overline{p}(b)>\overline{p}(a+b)$ for all integers $a,b>1$ by using an analytic result of Engle. We offer in this paper a combinatorial proof to the Liu-Zhang inequaity. More precisely, motivated by the polynomials $P_{n}(x)$ , which generalize the $k$-colored partitions function $p_{-k}(n)$, we introduce the polynomials $\overline{P}_{n}(x)$, which take the number of $k$-colored overpartitions of $n$ as their special values. And by combining combinatorial and analytic approaches, we obtain that
$\overline{P}_{a}(x) \overline{P}_{b}(x)>\overline{P}_{a+b}(x)$ for all positive integers $a,b$ and real numbers $x \ge 1$ , except for $(a,b,x)=(1,1,1),(2,1,1),(1,2,1)$.
\end{abstract}

\maketitle
\section{introduction}
A partition \cite{And} of a positive integer $n$ is a non-increasing sequence of positive numbers whose sum is $n$. Let $p(n)$ denote the number of partitions of $n$. Bessenrodt and Ono \cite{Ala} obtained that
\[ p(a) p(b)>p(a+b)\]
holds for $a,b>1$ and $a+b>9$.\par
The work of Bessenrodt and Ono promoted further research and new results in several directions. Shortly after their paper was published, Chern, Fu, and Tang \cite{Che} generalized Bessenrodt and Ono's theorem to $k$-colored partitions function.
A partition is called a $k$-colored partition if each part can appear as $k$ colors. Let $p_{-k}(n)$ denote the number of $k$-colored partitions of $n$. The generating function of $p_{-k}(n)$ is given by
\[\sum_{n=0}^{\infty} p_{-k}(n) q^{n}=\frac{1}{\prod_{n=1}^{\infty}\left(1-q^{n}\right)^{k}}=\frac{1}{(q ; q)_{\infty}^{k}},\]
where$(a ; q)_{\infty}=\prod_{n=0}^{\infty}\left(1-aq^{n}\right)$ is the $q$-Pochhammer symbol.
\begin{theorem}[Chern, Fu, Tang]
If $a,b,k$ are positive integers. Let $k>1$, then
\[ p_{-k}(a) p_{-k}(b)>p_{-k}(a+b),\]
except for $(a,b,k)={(1,1,2),(1,2,2),(2,1,2),(1,3,2),(3,1,2),(1,1,3)}.$
\end{theorem}
Similarly, the Bessenrodt and Ono's theorem was also extended to the overpartition function by Liu and Zhang. Recall an overpartition\cite{Co} of a nonnegative integer $n$ is a partition of $n$ where the last occurrence of each distinct part may be overlined. Let $\overline{p}(n)$ denote the number of overpartitions of $n$.
\begin{theorem}[Liu, Zhang]\label{th0}
If $a,b$ are integers with $a,b>1$, then
\[\overline{p}(a) \overline{p}(b)>\overline{p}(a+b).\]
\end{theorem}
The proof depends on a result of Engle\cite{En} on the bound of the overpartition function. In this paper, we will provide a combinatorial proof to the Liu-Zhang inequality for overpartition function, which closely follows the proof of the corresponding theorem for $p(n)$ in \cite{Ala}, and the proof method gives an extension of Theorem \ref{th0}. Finally, we obtain the following theorem.
\begin{theorem}\label{th1}
If $a,b$ are positive integers with $a\ge b$, then
\begin{align}\label{ie1}
\overline{p}(a) \overline{p}(b)>\overline{p}(a+b),
\end{align}
except for $(a,b)=(1,1),(2,1).$
\end{theorem}
\begin{remark}
For the two-tuples $(a,b)=(1,1),(2,1)$, we have equality in (\ref{ie1}).
\end{remark}
Heim and Neuhauser\cite{Hel} generalized the Chern-Fu-Tang Theorem to D'Arcais polynomials, also known as Nekrasov-Okounkov polynomials\cite{Han,Neu,Hel,Ne}. Let 
\[P_{n}(x):=\frac{x}{n}\sum_{k=1}^{n} \sigma(k) P_{n-k}(x),\]
with $\sigma(k):=\sum_{d \mid k} d$ is the divisor sum, and the initial condition $P_{0}(x):=1$. And 
\[\sum_{n=0}^{\infty} P_{n}(x) q^{n} = \prod_{n=1}^{\infty}\left(1-q^{n}\right)^{-x}, \text{ where } \lvert q \rvert <1,\] then $P_{n}(k)=p_{-k}(n)$ for all positive integers $k$, and $p(n)=p_{-1}(n)=P_{n}(1).$
\begin{theorem}[Heim, Neuhauser]
Let $a,b$ be positive integers with $a+b>2$ and $x>2$. Then 
\[P_{a}(x) P_{b}(x)>P_{a+b}(x).\]
The case $x=2$ is true for $a+b>4.$
\end{theorem}
Following the work of Heim and Neuhauser, and relating to the Liu and Zhang's work on the overpartiton function, we recursively define a family of polynomials.
\begin{Def}
Let $n$ be a positive integer, and denote $n$ by $n=2^ml$, where $l$ is odd. Let $\overline{\sigma}(n):=2^{m+1} \sigma(l), \overline{P}_{0}(x):=1$ and
\[\overline{P}_{n}(x):=\frac{x}{n}\sum_{k=1}^{n} \overline{\sigma}(k) \overline{P}_{n-k}(x).\]
\end{Def}
We have $\overline{P}_{0}(x)=1, \overline{P}_{1}(x)=2x, \overline{P}_{2}(x)=2x^2+2x, \overline{P}_{3}(x)=\dfrac{4x(x^2+3x+2)}{3},$ and $\overline{P}_{n}(x)$ is a polynomial of degree $n$ with positive coefficients.
\begin{theorem}\label{th2}
The generating function of $\overline{P}_{n}(x)$ can be expressed by
\[\sum_{n=0}^{\infty} \overline{P}_{n}(x)q^{n}=(\prod_{n=1}^{\infty}\frac{1+q^n}{1-q^{n}})^x,\]
and the derivatives $\overline{P}_{n}^{\prime}(x)$ of $\overline{P}_{n}(x)$ can be expressed by
\[\overline{P}_{n}^{\prime}(x)=\sum_{k=1}^{n} \frac{\overline{\sigma}(k)}{k} \overline{P}_{n-k}(x).\]
\end{theorem}
Similar with the polynomials $P_{n}(x)$, we can obtain the following inequalities for $\overline{P}_{n}(x)$.
\begin{theorem}\label{th3}
Let $n$ be a positive integer and $x \in \R$ with $x\ge 1.$ Then
\[\overline{P}_{n}(x)<\overline{P}_{n+1}(x) \quad \text { and } \quad 2 \leq \overline{P}_{n}^{\prime}(x)<\overline{P}_{n+1}^{\prime}(x).\]
Moreover, if $n+1=2^{s}, s>1$, then there exists $x_{n}\in (0,1)$ such that 
\begin{equation}\label{iequality1}
\overline{P}_{n+1}\left(x_{n}\right)<\overline{P}_{n}\left(x_{n}\right).
\end{equation}

\end{theorem}
\begin{theorem}\label{th4}
Let $a\ge b$ be positive integers and $x\ge 1$, then
\begin{equation}\label{ie2}
\overline{P}_{a}(x) \overline{P}_{b}(x)>\overline{P}_{a+b}(x),
\end{equation}
except for $(a,b,x)=(1,1,1),(2,1,1)$.
\end{theorem}
The theorem is proved by induction, and the proof contains flexible applications to a formula for the derivative of $\overline{P}_{n}(x)$, a bound of  $\overline{p}(n)$  (\ref{ie7}), Theorem \ref{th1} and Theorem \ref{th3}. \par
The numerical values in Table \ref{tab1} for $1\le a,b \le 10$ suggest that inequality (\ref{ie2}) does not always hold for all real numbers $0<x<1$, but for values of $x$ larger than the numbers(rounded to two decimal places) in position $(a,b)$ inequality (\ref{ie2}) holds.\par
By the generating function of $\overline{P}_{n}(x)$, we consider the $k$-colored overpartition function.
A partition is called a $k$-colored overpartition if each part can appear as $k$ colors, and the last occurrence of each part with distinct sizes or colors may be overlined. Let $\overline{p}_{-k}(n)$ denote the number of $k$-colored overpartitions of $n$.
\begin{example}
The number of $2$-colored overpartitions of $n=2$ are $\overline{p}_{-2}(2)=12: 2_{2}, 2_{1}, \overline{2}_{2}, \overline{2}_{1}, 1_{2}+1_{1}, 1_{2}+1_{2}, 1_{1}+1_{1}, 1_{1}+\overline{1}_{1}, \overline{1}_{2}+1_{1}, 1_{2}+\overline{1}_{1}, 1_{2}+\overline{1}_{2}, \overline{1}_{2}+\overline{1}_{1}.$
\end{example}
\noindent The generating function of $\overline{p}_{-k}(n)$ is given by
\[\sum_{n=0}^{\infty} \overline{p}_{-k}(n) q^{n}=(\prod_{n=1}^{\infty}\frac{1+q^n}{1-q^{n}})^k.\]
Then, we have $\overline{p}_{-k}(n)=\overline{P}_{n}(k)$ for positive integers $k$, in particular, $\overline{P}_{n}(1)=\overline{p}(n)$. Thus, Theorem \ref{th1} is the case $x=1$ of Theorem \ref{th4}, and for the case $x>1$,
we can obtain the following inequality for $k$-colored overpartitions function.
\begin{theorem}\label{th5}
If $a,b$ are positive integers and $k\ge 2$, we have
\[\overline{p}_{-k}(a) \overline{p}_{-k}(b)>\overline{p}_{-k}(a+b).\]
\end{theorem}
This paper is organized as follows. In Sect.\ref{sec2}, we present a combinatorial proof to the Liu-Zhang inequality, and we introduce the polynomials $\overline{P}_{n}(x)$ in Sect.\ref{sec3}, we show Theorem \ref{th3} and Theorem \ref{th4} in the next two sections. In the last section, we conclude an inequality for $k$-colored overpartitons function and give a combinatorial proof of Theorem \ref{th5}.
\section{Proof of Theorem \ref{th1}}\label{sec2}
In the section we present a combinatorial proof of Theorem \ref{th1}. We introduce some notations that will be used later. We denote the number of overpartitions of $n$ that satisfy a given condition by $\overline{p}\left(n \mid \text {condition }\right)$, while the enumerated set will be denoted by $\overline{P}\left(n \mid \text {condition }\right).$ \par The Cartesian product of two sets of partitions $A$ and $B$, denoted as $A \oplus B$, is given by
\[A \oplus B=\left\{(\lambda ; \mu ) \mid \lambda \in A, \mu \in B \right\}.\]
And we use the superscript to refer multiplicities. We arrange the parts in a unique way such that their sizes are non-increasing and the part which is overlined is no greater than the same numerical part which is non-overlined. For instance, $(4, \overline{4}, 2, \overline{2})$ is a overpartition of $12$ written in the standard way.\par
To show Theorem \ref{th1}, we need the following lemmas.
\begin{lemma}\label{le11}
Let $a,b$ be integers with $a\ge b \ge 2$, then
\[ \overline{p}\left(a \mid \textup{no } 1^{\prime} \mathrm{s}\right) \overline{p}\left(b \mid \textup{no } 2^{\prime} \mathrm{s}\right) > \overline{p}\left(a+b \mid \textup{no } 1^{\prime} \mathrm{s} \textup{ and no } 2^{\prime} \mathrm{s}\right),\]
\end{lemma}
\begin{remark}
The condition ``no $i^{\prime} \mathrm{s}$'' means that all parts of each partition are not equal to $i$, but one of them may be equal to $\overline{i}$, where $i=1,2.$
\end{remark}
\begin{proof}
For $\lambda=\left(\lambda_{1}, \lambda_{2}, \cdots, \lambda_{t}\right) \in \overline{P}(a+b),$ let
\[ i=i(\lambda)=\max \left\{j \in \mathbb{N} \mid 1 \leq j \leq t, \lambda_{j}+\cdots+\lambda_{t} \geq b\right\}. \]
Moreover, let $\lambda_{i}=x+y(x=x(\lambda), y=y(\lambda))$ such that
\[ x+\lambda_{i+1}+\cdots+\lambda_{t}=b \quad \text { and } \quad y+\lambda_{1}+\cdots+\lambda_{i-1}=a. \]
Note that $0<x \leq \lambda_{i}.$ Now define a map
 \[f: \overline{P}\left(a+b \mid \text {no } 1^{\prime} \mathrm{s} \text { and no } 2^{\prime} \mathrm{s}\right) \rightarrow \overline{P}\left(a \mid \text {no } 1^{\prime} \mathrm{s}\right) \oplus \overline{P}\left(b \mid \text {no } 2^{\prime} \mathrm{s}\right) \]
as follows. For $\lambda=\left(\lambda_{1}, \lambda_{2}, \cdots, \lambda_{t}\right) \in \overline{P}\left(a+b \mid \text {no } 1^{\prime} \mathrm{s} \text { and no } 2^{\prime} \mathrm{s}\right),$
\begin{equation*}
f(\lambda):= \begin{cases}\left(\lambda_{1}, \lambda_{2}, \cdots, \lambda_{i-1} ; \lambda_{i}, \cdots, \lambda_{t}\right), & \text { if } y=0 ; \\
\left(\lambda_{1}, \lambda_{2}, \cdots, \lambda_{i-1}, \overline{y}; \lambda_{i+1}, \cdots, \lambda_{t}, 1^{x}\right), & \text { if } y \geq 1, \lambda_{i} \text { is overlined }; \\
\left(\lambda_{1}, \lambda_{2}, \cdots, \lambda_{i-1}, y ; \lambda_{i+1}, \cdots, \lambda_{t}, 1^{x}\right), & \text { if } y \geq 2, \lambda_{i} \text { is non-overlined } ; \\
(\lambda_{1}, \ldots, \lambda_{i-2},\left\lceil\frac{\lambda_{i-1}+1}{2}\right\rceil,\left\lfloor\frac{\lambda_{i-1}+1}{2}\right\rfloor ; \\
  \lambda_{i+1}, \ldots, \lambda_{t}, 1^{x}), & \text { if } y = 1, \lambda_{i} \text { and } \lambda_{i-1} \text { are non-overlined } ; \\
& \\
(\lambda_{1}, \ldots, \lambda_{i-2},\overline{\left\lceil\frac{\lambda_{i-1}+1}{2}\right\rceil},\left\lfloor\frac{\lambda_{i-1}+1}{2}\right\rfloor ; \\
  \lambda_{i+1}, \ldots, \lambda_{t}, 1^{x}), & \text { if } y = 1, \lambda_{i} \text { is non-overlined },\\
& \text { and } \lambda_{i-1} \text { is overlined }. \end{cases}
\end{equation*}
Here, for convenience, in the cases $y\neq 0$ and $\lambda_{t}=\overline{1}$, we denote the partitions $(\lambda_{i+1}, \cdots, \lambda_{t-1}, 1^{x}, \lambda_{t})$ of $b$ by $(\lambda_{i+1}, \cdots, \lambda_{t-1}, \lambda_{t}, 1^{x})$ in the images of $f$, and for the cases $y=1, \lambda_{i-1}$ is odd and overlined and $\lambda_{i}$ is non-overlined, write the partitions $(\lambda_{1}, \ldots, \lambda_{i-2}, \left\lfloor\frac{\lambda_{i-1}+1}{2}\right\rfloor, \overline{\left\lceil\frac{\lambda_{i-1}+1}{2}\right\rceil})$ of $a$ by $(\lambda_{1}, \ldots, \lambda_{i-2},\overline{\left\lceil\frac{\lambda_{i-1}+1}{2}\right\rceil},\left\lfloor\frac{\lambda_{i-1}+1}{2}\right\rfloor)$. Since $a \ge 2$ and $y+\lambda_{1}+\cdots+\lambda_{i-1}=a$, then $\lambda_{i-1}$ does not vanish when $y=1$. Furthermore, since we arrange the parts of each partition in a unique way, it follows that $\lambda_{i-1}\ge \lambda_{i}$ if $\lambda_{i-1}$ exists. Moreover, for the case of $y=1$ and $\lambda_{i}$ is non-overlined, since $\lambda_{j}\neq 1,2(1\le j \le t),$ we must have $\lambda_{i}\ge 3$, which implies that $x=\lambda_{i}-y\ge 2$ and $\left\lfloor\frac{\lambda_{i-1}+1}{2}\right\rfloor \ge 2$. Thus, $f$ is well-defined.  In the third case, if $\lambda_{i-1}$ exists, we must have $\lambda_{i-1}-y\ge \lambda_{i}-y=x>0$. Notice that $\left\lceil\frac{\lambda_{i-1}+1}{2}\right\rceil- \left\lfloor\frac{\lambda_{i-1}+1}{2}\right\rfloor \in \left\{0,1\right\}$. To show the images of the third case are different from the fourth and fifth cases, it suffices to consider $\lambda_{i-1}-y=1$ in the third case. Since $0 < x=\lambda_{i}-y\le \lambda_{i-1}-y=1$, then $x=1$ in such situation, while $x\ge 2$ in the fourth and fifth cases. Thus, $f$ is injective. Since $b\ge 2$, this leads to $\lambda_{i-1}\neq \overline{1}$ for the case $y=0$. On the other hand, in the second case, it is easy to see that $\lambda_{i-1}\neq \overline{2}$ when $y=1$, then the pair $\left(\cdots, \overline{2}, \overline{1} ; \cdots \right)$ is not in the images of $f$ for $a\ge 3$. And for $a=2$, since $a\ge b\ge 2$, it involves that $b=2$, and $(\overline{2};\overline{2})$ is not the image of $f$. Therefore, $f$ is one-to-one but not onto.
\end{proof}
\begin{lemma}\label{le2}
If $a$ is a positive integer, then
\[\overline{p}\left(a \mid \textup {no } 1^{\prime} \mathrm{s}\right) \overline{p}\left(1 \right) \ge \overline{p}\left(a+1 \mid \textup {no } 1^{\prime} \mathrm{s}\right).\]
Moreover, if $a\ge 3$, then
\[\overline{p}\left(a \mid \textup {no } 1^{\prime} \mathrm{s}\right) \overline{p}\left(1 \right) >
\overline{p}\left(a+1 \mid \textup {no } 1^{\prime} \mathrm{s}\right).\]
\end{lemma}
\begin{proof}
Given a partition $\lambda=\left(\lambda_{1}, \lambda_{2}, \cdots, \lambda_{t}\right) \in \overline{P}\left(a+1 \mid \text {no } 1^{\prime} \mathrm{s}\right).$
Define a map
\[g_{1}: \overline{P}\left(a+1 \mid \text {no } 1^{\prime} \mathrm{s}\right) \rightarrow \overline{P}\left(a \mid \text {no } 1^{\prime} \mathrm{s}\right) \oplus \overline{P}\left(1 \right) \]
by
\begin{equation*}
g_{1}(\lambda):= \begin{cases} (\lambda_{1}, \lambda_{2}, \cdots, \lambda_{t-1},\overline{\lambda_{t}-1} ; 1), & \text { if } \lambda_{t} \geq 2, \lambda_{t} \text { is overlined }; \\
(\lambda_{1}, \lambda_{2}, \cdots, \lambda_{t-1},\lambda_{t}-1  ; 1), & \text { if } \lambda_{t} \ge 3, \lambda_{t} \text { is non-overlined }; \\
(\lambda_{1}, \lambda_{2},  \cdots , \lambda_{t-1}, \overline{1} ; \overline{1}), & \text { if } \lambda_{t}=2; \\
(\lambda_{1}, \lambda_{2}, \cdots, \lambda_{t-1} ; \lambda_{t}), & \text { if } \lambda_{t}=\overline{1}. \end{cases}
\end{equation*}
Here, $g_{1}$ is clearly well-defined. And by the definition of overpatition, in the case of $\lambda_{t}=\overline{1}$, we must have $\lambda_{t-1}\neq \overline{1},$ so the images of the third and fourth cases are disjoint. Therefore, $g_{1}$ is one-to-one, then
\[\overline{p}\left(a \mid \text {no } 1^{\prime} \mathrm{s}\right) \overline{p}\left(1 \right) \ge \overline{p}\left(a+1 \mid \text {no } 1^{\prime} \mathrm{s}\right).\]
Furthermore, if $a\ge3$, since $\lambda_{t-1} \neq \overline{2}$ for the cases $\lambda_{t}=2$, then
the pair $\left(\cdots, \overline{2}, \overline{1} ; \overline{1} \right)$ is not in the images of $g_{1}$, but $\left(\cdots, \overline{2}, \overline{1} ; \overline{1} \right) \in \overline{P}\left(a \mid \text {no } 1^{\prime} \mathrm{s}\right) \oplus \overline{P}\left(1 \right)$. Therefore, for $a\ge 3$,
\[\overline{p}\left(a \mid \text {no } 1^{\prime} \mathrm{s}\right) \overline{p}\left(1 \right) >
\overline{p}\left(a+1 \mid \text {no } 1^{\prime} \mathrm{s}\right).\]
\end{proof}
\begin{lemma}\label{le13}
If $a$ is an integer with $a\ge 2$, then
\[\overline{p}\left(a \mid \textup {no } 1^{\prime} \mathrm{s}\right) \overline{p}\left(2 \right) > \overline{p}\left(a+2 \mid \textup {no } 1^{\prime} \mathrm{s}\right).\]
\end{lemma}
\begin{proof}
Given a partition $\lambda=\left(\lambda_{1}, \lambda_{2}, \cdots, \lambda_{t}\right) \in \overline{P}\left(a+2 \mid \text {no } 1^{\prime} \mathrm{s}\right).$
Define a map
\[g_{2}: \overline{P}\left(a+2 \mid \text {no } 1^{\prime} \mathrm{s}\right) \rightarrow \overline{P}\left(a \mid \text {no } 1^{\prime} \mathrm{s}\right) \oplus \overline{P}\left(2 \right) \]
by
\begin{equation*}
g_{2}(\lambda):= \begin{cases} (\lambda_{1}, \lambda_{2}, \cdots, \lambda_{t-1},\overline{\lambda_{t}-2} ; 2), & \text { if } \lambda_{t} \geq 3, \lambda_{t} \text { is overlined }; \\
(\lambda_{1}, \lambda_{2}, \cdots, \lambda_{t-1},\lambda_{t}-2  ; 2), & \text { if } \lambda_{t} \ge 4, \lambda_{t} \text { is non-overlined }; \\
(\lambda_{1}, \lambda_{2},  \cdots , \lambda_{t-1}, \overline{1} ; \overline{2}), & \text { if } \lambda_{t}=3; \\
(\lambda_{1}, \lambda_{2}, \cdots, \lambda_{t-1} ; 1,1), & \text { if } \lambda_{t}=2;\\
(\lambda_{1}, \lambda_{2}, \cdots, \lambda_{t-1} ; \lambda_{t}), & \text { if } \lambda_{t}=\overline{2};\\
(\lambda_{1}, \lambda_{2}, \cdots, \lambda_{t-2},\overline{\lambda_{t-1}-1} ; 1, \overline{1}), & \text { if } \lambda_{t} = \overline{1}, \lambda_{t-1} \text { is overlined }; \\
(\lambda_{1}, \lambda_{2}, \cdots, \lambda_{t-2},  \overline{1} ; 1,1), & \text { if } \lambda_{t}=\overline{1}, \lambda_{t-1}=2;\\
(\lambda_{1}, \lambda_{2}, \cdots, \lambda_{t-2},\lambda_{t-1}-1  ; 1, \overline{1}), & \text { if } \lambda_{t} =\overline{1}, \lambda_{t-1}\ge 3 \text{ and }\lambda_{t-1} \text { is non-overlined }. \end{cases}
\end{equation*}
Here, $g_{2}$ is clearly well-defined. Note that $\lambda_{t-1}\ge 2$ for the cases $\lambda_{t}=\overline{1},2,\overline{2}$, so the images for all the cases are mutually exclusive, it follows that $g_{2}$ is one-to-one.
Furthermore, it is clear that the pairs $\left(\cdots, \overline{2}, \overline{1} ; \cdots \right)$ are not in the images of $g_{2}$, but $\left(\cdots, \overline{2}, \overline{1} ; \cdots \right) \in \overline{P}\left(a \mid \text {no } 1^{\prime} \mathrm{s}\right) \oplus \overline{P}\left(2 \right)$ for $a\ge 3$. For $a=2$, it can be easily checked that the inequality holds. Hence, the result is as required.
\end{proof}
\begin{lemma}
If $a,b$ are integers with $a\ge b\ge 1$, then
\[\overline{p}\left(a \mid \textup {no } 1^{\prime} \mathrm{s}\right) \overline{p}(b) > \overline{p}\left(a+b \mid \textup {no } 1^{\prime} \mathrm{s}\right),\]
except for $(a,b)=(1,1),(2,1).$
\end{lemma}
\begin{proof}
For $b=1,2$, we have proved in Lemma \ref{le2} and Lemma \ref{le13}. So we assume that $b \ge 3$. Let $n=a+b$, we will apply induction on $n$.\par
Base case($a=b=3,n=6$),
\[\overline{p}\left(3 \mid \textup {no } 1^{\prime} \mathrm{s}\right)=4, \text{ } \overline{p}(3)=8, \text{ } \overline{p}\left(6 \mid \textup {no } 1^{\prime} \mathrm{s}\right)=16,\]
thus, \[\overline{p}\left(3 \mid \textup {no } 1^{\prime} \mathrm{s}\right) \overline{p}(3) > \overline{p}\left(6 \mid \textup {no } 1^{\prime} \mathrm{s}\right).\]\par
Inductive step: suppose the inequality is true for smaller sum $a+b$. Then by the inductive hypothesis and Lemma \ref{le11}, Lemma \ref{le2} and Lemma \ref{le13}, we have
\begin{align}
 \overline{p}\left(a+b \mid \text {no } 1^{\prime} \mathrm{s}\right)={} & \overline{p}\left(a+b \mid \text {no } 1^{\prime} \mathrm{s} \text { and at least one } 2^{\prime} \mathrm{s}\right)+ \overline{p}\left(a+b \mid \text {no } 1^{\prime} \mathrm{s} \text { and no } 2^{\prime} \mathrm{s}\right) \notag \\
={} & \overline{p}\left(a+b-2 \mid \text {no } 1^{\prime} \mathrm{s}\right)+\overline{p}\left(a+b \mid \text {no } 1^{\prime} \mathrm{s} \text { and no } 2^{\prime} \mathrm{s}\right) \notag \\
< {} & \overline{p}\left(a \mid \text {no } 1^{\prime} \mathrm{s}\right) \overline{p}(b-2)+\overline{p}\left(a+b \mid \text {no } 1^{\prime} \mathrm{s} \text { and no } 2^{\prime} \mathrm{s}\right) \notag \\
< {} &  \overline{p}\left(a \mid \text {no } 1^{\prime} \mathrm{s}\right) \overline{p}\left(b \mid \text {at least one } 2^{\prime} \mathrm{s}\right)+\overline{p}\left(a \mid \text {no } 1^{\prime} \mathrm{s}\right) \overline{p}\left(b \mid \text {no } 2^{\prime} \mathrm{s}\right) \notag \\
={} & \overline{p}\left(a \mid \text {no } 1^{\prime} s\right) \overline{p}(b). \notag
\end{align}
This completes the proof.
\end{proof}

Next, we are ready to prove Theorem \ref{th1}.
\begin{proof}[Proof of Theorem \ref{th1}]
We apply induction on $n=a+b$.\\
Base case: It can be verified that the inequality holds for $n=4(a=3,b=1$ or $a=2,b=2)$.\\
Inductive step: Assume that $n\ge 5$ and the inequality holds for $n-1$. Since $a\ge b$, we have $a\ge 3$. Thus, by inductive hypothesis,
\[\overline{p}(a+b-1)<\overline{p}(a-1)\overline{p}(b).\]
Then
\begin{align}
\overline{p}(a+b)= {} & \overline{p}\left(a+b \mid \text {at least one } 1^{\prime} \mathrm{s}\right)+\overline{p}\left(a+b \mid \text {no } 1^{\prime} \mathrm{s}\right) \notag \\
= {} & \overline{p}(a+b-1)+\overline{p}\left(a+b \mid \text {no } 1^{\prime} \mathrm{s}\right) \notag \\
< {} & \overline{p}(a-1) \overline{p}(b)+\overline{p}\left(a \mid \text {no } 1^{\prime} \mathrm{s}\right) \overline{p}(b) \notag \\
= {} & \overline{p}(a-1) \overline{p}(b)+\left\{\overline{p}(a)-\overline{p}\left(a \mid \text {at least one } 1^{\prime} \mathrm{s}\right)\right\} \overline{p}(b) \notag \\
= {} & \overline{p}(a-1) \overline{p}(b)+\left\{\overline{p}(a)-\overline{p}(a-1)\right\} \overline{p}(b) \notag \\
= {} & \overline{p}(a) \overline{p}(b). \notag
\end{align}
Therefore, by the principle of mathematical induction, the inequality holds for $n\ge 4$. And the exceptions can be individually checked. Thus, we complete the proof.
\end{proof}

\section{Proof of Theorem \ref{th2}}\label{sec3}
\begin{proof}[Proof of Theorem \ref{th2}]
By the definition, we have $\overline{P}_{0}(x):=1$ and
\[\overline{P}_{n}(x)=\frac{x}{n}\sum_{k=1}^{n} \overline{\sigma}(k) \overline{P}_{n-k}(x), \text{ for } n\ge 1,\]
then
\[ \sum_{n=1}^{\infty} \overline{P}_{n}(x) n q^{n-1}=\sum_{n=1}^{\infty} x \left(\sum_{k=1}^{n} \overline{\sigma}(k)  \overline{P}_{n-k}(x)\right)q^{n-1}.
\]
And the equality is equivalent to the following equality,
\begin{align}
\frac{\mathrm{d}}{\mathrm{d} q}\left(\sum_{n=0}^{\infty} \overline{P}_{n}(x) q^{n}\right)= {} & x\left(\sum_{n=0}^{\infty}\sum_{k=0}^{n} \overline{\sigma}(k+1)\overline{P}_{n-k}(x)q^{n}\right) \notag \\
= {} & \left(\sum_{n=0}^{\infty}\overline{P}_{n}(x) q^{n}\right)\left(x\sum_{n=1}^{\infty}\overline{\sigma}(n)q^{n-1}\right). \notag
\end{align}
Since $\overline{P}_{0}(x)=1$, it implies that
\begin{align}\label{ie3}
\sum_{n=0}^{\infty} \overline{P}_{n}(x) q^{n}=\exp \left(x \sum_{n=1}^{\infty} \overline{\sigma}(n) \frac{q^{n}}{n}\right).
\end{align}
On the other hand,
\begin{align}
\sum_{m=1}^{\infty} \ln(1+q^m)= {} & \sum_{m=1}^{\infty}\sum_{k=1}^{\infty}(-1)^{k-1}\frac{q^{mk}}{k} \notag \\
= {} & \sum_{n=1}^{\infty}\sum_{d \mid n}(-1)^{d-1}\frac{q^{n}}{d} \notag \\
= {} & \sum_{n=1}^{\infty}\frac{1}{n}\sum_{d \mid n}(-1)^{\frac{n}{d}-1}dq^{n} \notag \\
= {} & - \sum_{n=1}^{\infty}\frac{1}{n}\tau(n)q^{n}, \text{ where } \tau(n)=\sum_{d \mid n}(-1)^{\frac{n}{d}}d. \notag
\end{align}
Then
\[\prod_{n=1}^{\infty}\left(1+q^{n}\right)=\exp \left(-\sum_{n=1}^{\infty} \tau(n) \dfrac{q^{n}}{n}\right).\]
Similarly,
\[\prod_{n=1}^{\infty}\left(1-q^{n}\right)=\exp \left(-\sum_{n=1}^{\infty} \sigma(n) \dfrac{q^{n}}{n}\right).\]
Thus,
\[\prod_{n=1}^{\infty}\frac{1+q^n}{1-q^{n}}=\exp\left(\sum_{n=1}^{\infty} (\sigma(n)-\tau(n)) \dfrac{q^{n}}{n}\right).\]
Denoted $n$ by $n=2^ml$, where $l$ is odd. Then $\sigma(n)-\tau(n)=2^{m+1}\sigma(l)=\overline{\sigma}(n).$ \\
Therefore,
\begin{align}\label{ie4}
\sum_{n=0}^{\infty} \overline{P}_{n}(x)q^{n}=\exp \left(x \sum_{n=1}^{\infty} \overline{\sigma}(n) \frac{q^{n}}{n}\right)=(\prod_{n=1}^{\infty}\frac{1+q^n}{1-q^{n}})^x.
\end{align}
As for the expression of $\overline{P}_{n}^{\prime}(x),$ the proof is based on examining the logarithmic derivative of (\ref{ie3}), certain grouping of terms and using the uniqueness of the coefficients of the involved power series in $q$.\par
We differentiate both side of (\ref{ie3}) with respect to $x$, and make use of known property of Cauchy product, we can find that the formula(valid for $n\ge 1$)
\[\overline{P}_{n}^{\prime}(x)=\sum_{k=1}^{n} \frac{\overline{\sigma}(k)}{k} \overline{P}_{n-k}(x).\]
\end{proof}

\section{Proof of Theorem \ref{th3}}
\begin{proof}[Proof of Theorem \ref{th3}]
Let $\Delta_{n}(x):=\overline{P}_{n+1}(x)-\overline{P}_{n}(x)$ for $n\ge 1.$
We prove Theorem \ref{th3} by induction on $n$.\\
Claim: $\overline{P}_{n}(x)<\overline{P}_{n+1}(x)$ for all $n \ge 1$ and $x \ge 1.$ \par
Let $n=1$, we have $\Delta_{1}(x)=2x^2 >0$ for all $x >0.$ Suppose $n\ge 2$ and  $\Delta_{m}(x)>0$ is true for all real numbers $x\ge 1$ and integers $1\le m \le n-1.$\par
By inductive hypothesis, we have $\Delta_{m}(x)=\overline{P}_{m+1}(x)-\overline{P}_{m}(x)>0$ for $1\le m \le n-1$ and $x\ge 1.$ Note that $P_{1}(x)=2x> 1=P_{0}(x)$ for $x \ge 1$, we can obtain
\[\sum_{k=1}^{n} \dfrac{\overline{\sigma}(k)}{k} \overline{P}_{n+1-k}(x) > \sum_{k=1}^{n} \dfrac{\overline{\sigma}(k)}{k} \overline{P}_{n-k}(x).\]
Thus,
\begin{align}
\overline{P}_{n+1}^{\prime}(x)= {} & \sum_{k=1}^{n+1} \frac{\overline{\sigma}(k)}{k} \overline{P}_{n+1-k}(x)  \notag \\
= {} & \sum_{k=1}^{n} \frac{\overline{\sigma}(k)}{k} \overline{P}_{n+1-k}(x)+\frac{\overline{\sigma}(n+1)}{n+1}  \notag \\
> {} & \sum_{k=1}^{n} \frac{\overline{\sigma}(k)}{k} \overline{P}_{n-k}(x)=\overline{P}_{n}^{\prime}(x). \notag
\end{align}
That is, $\overline{P}_{n+1}^{\prime}(x)>\overline{P}_{n}^{\prime}(x)$ for $x\ge 1$.\\
Since $\overline{P}_{1}(x)=2x$, it follows that
\[2\le \overline{P}_{n}^{\prime}(x)<\overline{P}_{n+1}^{\prime}(x) \text{ for } x\ge 1.\]
It is well known that $\overline{p}(n)$ increases strictly, so we have
\[\overline{P}_{n+1}(1)=\overline{p}(n+1)>\overline{p}(n)=\overline{P}_{n}(1),\]
that is $\Delta_{n}(1)>0$. And $\Delta_{n}^{\prime}(x)=\overline{P}_{n+1}^{\prime}(x)-\overline{P}_{n}^{\prime}(x)>0$ for $x\ge 1$, which yield that $\Delta_{n}(x)\ge \Delta_{n}(1)>0,$ thus $\overline{P}_{n+1}(x)>\overline{P}_{n}(x)$.\par
Next, we show (\ref{iequality1}). Since polynomials $\overline{P}_{n}(x)$ have positive coefficients with degree $n$, then 
\[\lim _{x \rightarrow \infty} \Delta_{n}(x)=+\infty.\] 
And note that $\Delta_{n}(0)=0$. We only need to prove $\Delta_{n}^{\prime}(0)<0$ for $n+1=2^s,s>1.$ This is true that
\[\Delta_{n}^{\prime}(0)=\frac{\overline{\sigma}(n+1)}{n+1}-\frac{\overline{\sigma}(n)}{n},\]
and $\overline{\sigma}(n+1)=2(n+1), \overline{\sigma}(n)\ge 2n+2$.
\end{proof}
\begin{table}[h]
\begin{center}
    \caption{ Maximal positive real roots of $\overline{P}_{a, b}(x)$ }\label{tab1}
\setlength{\tabcolsep}{4mm}{
\begin{tabular}{ccccccccccc}
    \toprule
        $x_{a,b}$ & 1 & 2 & 3 & 4 & 5 & 6 & 7 & 8 & 9 & 10 \\
    \midrule
        1 & 1.00 & 1.00 & 0.80 & 0.81 & 0.78 & 0.74 & 0.72 & 0.72 & 0.70 & 0.69 \\
        2 & 1.00 & 0.84 & 0.70 & 0.70 & 0.65 & 0.61 & 0.60 & 0.59 & 0.57 & 0.56 \\
        3 & 0.80 & 0.70 & 0.57 & 0.56 & 0.51 & 0.48 & 0.47 & 0.46 & 0.44 & 0.43 \\
        4 & 0.81 & 0.70 & 0.56 & 0.54 & 0.51 & 0.47 & 0.46 & 0.45 & 0.43 & 0.42 \\
        5 & 0.78 & 0.65 & 0.51 & 0.51 & 0.47 & 0.43 & 0.42 & 0.41 & 0.39 & 0.39 \\
        6 & 0.74 & 0.61 & 0.48 & 0.47 & 0.43 & 0.40 & 0.39 & 0.38 & 0.36 & 0.35 \\
        7 & 0.72 & 0.60 & 0.47 & 0.46 & 0.42 & 0.39 & 0.38 & 0.37 & 0.35 & 0.34 \\
        8 & 0.72 & 0.59 & 0.46 & 0.45 & 0.41 & 0.38 & 0.37 & 0.36 & 0.34 & 0.33 \\
        9 & 0.70 & 0.57 & 0.44 & 0.43 & 0.39 & 0.36 & 0.35 & 0.34 & 0.32 & 0.31 \\
        10 & 0.69 & 0.56 & 0.43 & 0.42 & 0.39 & 0.35 & 0.34 & 0.33 & 0.31 & 0.30 \\
    \bottomrule
    \end{tabular}}
\end{center}
\end{table}

\section{Proof of Theorem \ref{th3}}
\noindent Let $a,b$ be positive integers, and $\overline{P}_{a, b}(x):=\overline{P}_{a}(x) \overline{P}_{b}(x)-\overline{P}_{a+b}(x)$.
Then
$\overline{P}_{a, b}(0)=0, \overline{P}_{a, b}^{\prime}(0)=-\dfrac{\overline{\sigma}{(a+b)}}{a+b},$ and
\[\lim _{x \rightarrow \infty} \overline{P}_{a, b}(x)=\infty.\]\par
We are especially interested in the largest non-negative real root $x_{a,b}$ of $\overline{P}_{a, b}(0)$, since 
$\overline{P}_{a}(x) \overline{P}_{b}(x)>\overline{P}_{a+b}(x)$ for $x>x_{a,b}$.
Table \ref{tab1} records these roots for $1\le a, b \le 10,$ and the roots distribution in Table \ref{tab1} explains exceptions in Theorem \ref{th1}.\par
Let $\mu:=\mu(n)=\pi\sqrt{n}$. Zukermann\cite{Zu} gave a formula for the overpartition function, which is indeed a Rademacher-type convergent series,
\begin{equation}\label{ie6}
\overline{p}(n)=\frac{1}{2 \pi} \sum_{\substack{k=1 \\ 2 \nmid k}}^{\infty} \sqrt{k} \sum_{\substack{h=0 \\(h, k)=1}}^{k-1} \frac{\omega(h, k)^{2}}{\omega(2 h, k)} e^{-\frac{2 \pi i n h}{k}} \frac{\mathrm{d}}{\mathrm{d} n}\left(\frac{\sinh \frac{\pi \sqrt{n}}{k}}{\sqrt{n}}\right),
\end{equation}
where
\[\omega(h, k)=\exp \left(\pi i \sum_{r=1}^{k-1} \frac{r}{k}\left(\frac{h r}{k}-\left\lfloor\frac{h r}{k}\right\rfloor-\frac{1}{2}\right)\right)\]
for positive integers $h$ and $k$. From this Rademacher-type series (\ref{ie6}), Engel\cite{En} provided an error term for the overpartition function
\[\overline{p}(n)=\frac{1}{2 \pi} \sum_{\substack{k=1 \\ 2 \nmid k}}^{N} \sqrt{k} \sum_{\substack{h=0 \\(h, k)=1}}^{k-1} \frac{\omega(h, k)^{2}}{\omega(2 h, k)} e^{-\frac{2 \pi i n h}{k}} \frac{\mathrm{d}}{\mathrm{d} n}\left(\frac{\sinh \frac{\mu}{k}}{\sqrt{n}}\right)+R_{2}(n, N),\]
where
\[\lvert R_{2}(n, 2)\rvert \leq \frac{2^{\frac{5}{2}}}{n \mu} \sinh \left(\frac{\mu}{2}\right).\]
Moreover, Engel\cite{En} proved that $\overline{p}(n)$ is logconcave for $n\ge 2$. For our purpose the case $N=2$, Liu and Zhang\cite{Liu} found that:
\begin{align}\label{ie7}
\frac{e^{\mu}}{8 n}\left(1-\frac{1}{\sqrt{n}}\right)<\overline{p}(n)<\frac{e^{\mu}}{8 n}\left(1+\frac{1}{n}\right) \text{ for } n\ge 1.
\end{align}
This is the effective estimate for $\overline{P}_{n}(1)$ that we will use in our proof.\\

To prove the theorem, we first establish a few lemmas, and we utilize the well-known upper bound:
\[\sigma{(m)}\le m(1+\ln(m)).\]
\begin{lemma}\label{le3}
Let $n\ge 1$ and $x\ge 1$, then
\[\overline{P}_{n}(x)-(1+\ln (2n))>0.\]
\end{lemma}
\begin{proof}
Since $\overline{P}_{n}(x)$ has non-negative coefficients, it is sufficient to show that
\[\overline{P}_{n}(1) > (1+\ln (2 n)) \text{ for } n\ge 1.\]
Suppose $n\ge 16,$ by (\ref{ie7}), we have
\begin{align}
\overline{P}_{n}(1)>\frac{1}{8 n}\left(1-\frac{1}{\sqrt{n}}\right) \mathrm{e}^{\pi \sqrt{n}}>\frac{\left(\pi \sqrt{n}\right)^{4}}{4 ! 8 n}\left(1-\frac{1}{\sqrt{n}}\right)
>\frac{\pi^4}{256}n.\notag
\end{align}
Since
\[1+\ln(2n)=1+\ln32+\ln(\frac{n}{16})\le \ln32+ \frac{n}{16},\]
for
\[n\ge 16> \frac{\ln(32)}{\dfrac{\pi^4}{256}-\dfrac{1}{16}},\]
we have $\overline{P}_{n}(1)>\ln32+ \dfrac{n}{16} \ge 1+\ln(2n).$
It remains to check $\overline{P}_{n}(1)>1+\ln(2n)$ for $1\le n\le 15.$
For the cases $n=1,2$, it can be individual checked. For the case $3\le n\le 15$, we have $1+\ln(2n)<5$. Since $\overline{P}_{n}(1)$ increases monotonously in $n$ by Theorem \ref{th3}, we also have $1+\ln(2n)<8=\overline{P}_{3}(1)\le \overline{P}_{n}(1)$ for $3\le n\le 15$.
\end{proof}
\begin{lemma}\label{le4}
Let $a$ be a positive integer, if there exists real number $x_{0}\ge 1$ such that $\overline{P}_{a+b}(x_{0})>(1+\ln (2 a)) \overline{P}_{b}(x_{0})$ for all integers $1\le b < a$, then
\[\overline{P}_{a+b}(x)>(1+\ln (2 a)) \overline{P}_{b}(x), \text{ for } x\ge x_{0} \text{ and } 1\le b < a.\]
\end{lemma}
\begin{proof}
Let $f(x)=\overline{P}_{a+b}(x)-(1+\ln (2 a)) \overline{P}_{b}(x)$. Since
\[\overline{P}_{n}^{\prime}(x)=\sum_{k=1}^{n} \frac{\overline{\sigma}(k)}{k} \overline{P}_{n-k}(x),\]
then the $k$-th derivative of $f(x)$ can be denoted by
\[f^{(k)}(x)=g_{k}(x)+\sum_{l=1}^{m} c_{k, l}(\overline{P}_{a+m-l}(x)-(1+\ln (2 a)) \overline{P}_{m-l}(x)),\]
where $m=b-k+1, 1\le k\le b, c_{k, l}>0$, and $g_{k}(x)$ is a linear combination of $\overline{P}_{n}(x)(0\le n \le a+b-k)$ with non-negative coefficients. Since $\overline{P}_{n}(x)$ has non-negative coefficients, then by hypothesis and Lemma \ref{le3}, we have $f^{(k)}(x_{0})>0$ for $1\le k \le b$. \par
Note that $f^{(b)}(x_{0})>0$, and we can also find that $f^{(b)}(x)$ increases for $x\ge 0$, then we have $f^{(b)}(x)>0$ for $x\ge x_{0}$, it implies that $f^{(b-1)}(x)$ increases for $x\ge x_{0}$. Since $f^{(b-1)}(x_{0})>0$, then $f^{(b-2)}(x)$ increases for $x\ge x_{0}$, continuing in this way, we obtain $f(x)$ increases for $x\ge x_{0}$. Since $f(x_{0})>0$, then $f(x)>0$ for $x\ge x_{0}$. That is,
\[\overline{P}_{a+b}(x)>(1+\ln (2 a)) \overline{P}_{b}(x), \text{ for } x\ge x_{0} \text{ and } 1\le b < a.\]
\end{proof}
\begin{proof}[Proof of Theorem \ref{th4}]
The case $x=1$ is proved in Theorem \ref{th1}. So we assume that $x>1$, and we show (\ref{ie2}) by induction on $n=a+b.$\\
For $n=a+b=2,$ we have $\overline{P}_{1}(x) \overline{P}_{1}(x)-\overline{P}_{2}(x)=2x(x-1) >0$ for $x>1 $.\\
Suppose $n\ge 3$ and $\overline{P}_{A}(x) \overline{P}_{B}(x)>\overline{P}_{A+B}(x)$ for $2 \leq A+B \leq n-1=a+b-1$ and $x> 1$.\\
Without loss of generality, we assume that $a\ge b$, then $a\ge 2$.\\
By Theorem \ref{th1}, we have
\[\overline{P}_{a}(1) \overline{P}_{b}(1) \ge \overline{P}_{a+b}(1), \text{ for } a+b\ge 2.\]
If we can show that $\dfrac{\mathrm{d}}{\mathrm{d} x}(\overline{P}_{a}(x) \overline{P}_{b}(x))>\overline{P}_{a+b}^{\prime}(x)$ for $x>1$,  which implies $\overline{P}_{a}(x) \overline{P}_{b}(x)>\overline{P}_{a+b}(x)$ for $x> 1$, the proof is completed.
Note that $\overline{P}_{A}(x) \overline{P}_{0}(x)=\overline{P}_{A}(x).$\\
Thus,
\begin{align*}
     & \overline{P}_{a}^{\prime}(x) \overline{P}_{b}(x)+\overline{P}_{a}(x) \overline{P}_{b}^{\prime}(x)-\overline{P}_{a+b}^{\prime}(x)  \notag \\
= {} & \sum_{k=1}^{a} \frac{\overline{\sigma}(k)}{k} \overline{P}_{a-k}(x) \overline{P}_{b}(x)+\overline{P}_{a}(x) \sum_{k=1}^{b} \frac{\overline{\sigma}(k)}{k} \overline{P}_{b-k}(x)-\sum_{k=1}^{a+b} \frac{\overline{\sigma}(k)}{k} \overline{P}_{a+b-k}(x) \notag \\
> {} & \sum_{k=1}^{a} \frac{\overline{\sigma}(k)}{k} \overline{P}_{a+b-k}(x)+\sum_{k=1}^{b} \frac{\overline{\sigma}(k)}{k} \overline{P}_{a+b-k}(x)-\sum_{k=1}^{a+b} \frac{\overline{\sigma}(k)}{k} \overline{P}_{a+b-k}(x) \notag \\
= {} & \sum_{k=1}^{b} \frac{\overline{\sigma}(k)}{k} \overline{P}_{a+b-k}(x)- \sum_{k=a+1}^{a+b}\frac{\overline{\sigma}(k)}{k} \overline{P}_{a+b-k}(x) \notag \\
= {} & \sum_{k=1}^{b} (\frac{\overline{\sigma}(k)}{k} \overline{P}_{a+b-k}(x)-\frac{\overline{\sigma}(k+a)}{k+a} \overline{P}_{b-k}(x)) \notag \\
\ge {} & \sum_{k=1}^{b} (\frac{2k}{k} \overline{P}_{a+b-k}(x)-\frac{2\sigma(k+a)}{k+a} \overline{P}_{b-k}(x)) \notag \\
\ge {} & 2\sum_{k=1}^{b} (\overline{P}_{a+b-k}(x)-(1+\ln (2 a)) \overline{P}_{b-k}(x)). \notag
\end{align*}
Let us now consider
\begin{align}\label{ie8}
\overline{P}_{a+b-k}(x)-(1+\ln (2 a)) \overline{P}_{b-k}(x)
\end{align}
for each $k$ separately.
We only need to show $\overline{P}_{a+b-k}(1)-(1+\ln (2 a)) \overline{P}_{b-k}(1)>0$.\\
Since
\[\frac{e^{\mu}}{8 n}\left(1-\frac{1}{\sqrt{n}}\right)<\overline{P}_{n}(1)<\frac{e^{\mu}}{8 n}\left(1+\frac{1}{n}\right) \text{ for } n\ge 1,\]
for $k<b$,
\begin{align}
     & \overline{P}_{a+b-k}(1)-(1+\ln (2 a)) \overline{P}_{b-k}(1)  \notag \\
> {} & \frac{1}{8(a+b-k)}\left(1-\frac{1}{\sqrt{a+b-k}}\right) \mathrm{e}^{\pi\sqrt{a+b-k}}\notag \\
  {} & -(1+\ln (2 a)) \frac{1}{8(b-k)}\left(1+\frac{1}{b-k}\right) \mathrm{e}^{\pi\sqrt{b-k}}\label{ie9},
\end{align}
and the last is positive if and only if
\begin{align}\label{ie10}
\mathrm{e}^{\pi(\sqrt{a+b-k}-\sqrt{b-k})}>(1+\ln (2 a)) \frac{a+b-k}{b-k} \frac{1+\frac{1}{b-k}}{1-\frac{1}{\sqrt{a+b-k}}}.
\end{align}
Since
\begin{align*}
     & \sqrt{a+b-k}-\sqrt{b-k}  \notag \\
= {} & \frac{a}{\sqrt{a+b-k}+\sqrt{b-k}}  \notag \\
\ge {} & \frac{a}{\sqrt{2 a-1}+\sqrt{a-1}}> \frac{a}{2 \sqrt{a}+\sqrt{a}}=\frac{\sqrt{a}}{3}, \notag
\end{align*}
on the other hand,
\[(1+\ln (2 a)) \frac{a+b-k}{b-k} \frac{1+\frac{1}{b-k}}{1-\frac{1}{\sqrt{a+b-k}}}< (1+\ln (2 a))(1+a) \frac{2}{1-\frac{1}{\sqrt{a}}}.\]
If $a$ satisfies
\begin{equation}\label{ie11}
\mathrm{e}^{\pi \sqrt{a} / 3}>(1+\ln (2 a))(1+a) \frac{2}{1-\frac{1}{\sqrt{a}}},
\end{equation}
then $a$ also fulfill (\ref{ie10}). Heim and Neuhauser\cite{Hel} showed that
\[\mathrm{e}^{\pi \sqrt{a} / 3}>(1+\ln (2 a))(1+a) \frac{2}{1-\frac{1}{\sqrt{a}}}\]
holds for $a\ge94.$
This means that (\ref{ie8}) is positive for $x= 1.$ Then Lemma \ref{le4} implies that (\ref{ie8}) is positive for $x\ge 1.$\par
We used JavaScript to check that (\ref{ie8}) with $x=1$ is positive for $1\le k <b\le a \le 93$. Again Lemma \ref{le4} implies that (\ref{ie8}) is positive for $x\ge 1$.
What remains to be considered is (\ref{ie8}) is positive for $k=b.$ This follows from Lemma \ref{le4} setting $n=a.$
Thus, for all $a\ge b\ge k \ge 1,$ the value of (\ref{ie8}) is positive. Hence,
\[\overline{P}_{a}^{\prime}(x) \overline{P}_{b}(x)+\overline{P}_{a}(x) \overline{P}_{b}^{\prime}(x)>\overline{P}_{a+b}^{\prime}(x), \text{ for } x>1.\]
This completes the induction step.
\end{proof}
\section{Combinatorial Proof of Theorem \ref{th5}}
Actually, Theorem \ref{th5} is a special case of Theorem \ref{th4}. And the proof of Theorem \ref{th4} is analytic and inductive. In this section we give a combinatorial proof of Theorem \ref{th5} that is largely motivated by the work of Alanazi, Gagola, and Munagi as well as Chern, Fu, and Tang. By the combinatorial proof, we can see the explicit one-to-one correspondence between the two sides of the inequality for $k$-colored overpartition function.\par
Firstly, we introduce some notations that will be used later. Following Andrews and Eriksson \cite{Eri}, we denote the number of $k$-colored overpartitions of $n$ that satisfy a given condition by $\overline{p}_{-k}\left(n \mid \text {condition }\right)$, while the enumerated set will be denoted by $\overline{P}_{-k}\left(n \mid \text {condition }\right).$ \par
And for $k$-colored overpartition, we use the subscript $1,2,...,k$ to indicate the color of a specific part, while the superscript refers to multiplicities, for instance, $3_{5}^2$ stands for two parts of size $3$ with color $5$. We arrange the parts in a unique way such that both of their sizes and colors are weakly decreasing and the part which is overlined is no greater than the same numerical part which is non-overlined. For instance, $(4_{3}, 4_{2}, \overline{4}_{2}, 2_{3}^2, \overline{2}_{3}, 2_{1}, \overline{2}_{1})$ is a $3$-colored overpartition of $22$ written in the standard way. To avoid heavy notation and clarify possible confusion, we point out that the meaning of $\lambda_{i}$ is the $i$th part of a partition $\lambda$, not a part of size $\lambda$ and color $i$.\\
To prove the theorem, we need a few lemmas as follows.
\begin{lemma}\label{le5}
If $a,b,k$ are positive integers with $a\ge b$ and $k\ge 2$, then
\begin{equation}\label{ie12}
\overline{p}_{-k}\left(a \mid \textup {no } 1_{1}^{\prime} \mathrm{s}\right) \overline{p}_{-k}\left(b \mid \textup {no } 1_{2}^{\prime} \mathrm{s}\right) > \overline{p}_{-k}\left(a+b \mid \textup {no } 1_{1}^{\prime} \mathrm{s} \textup { and no } 1_{2}^{\prime} \mathrm{s}\right).
\end{equation}
\end{lemma}
\begin{proof}
Firstly, consider the cases $(a,b,k)=(1,1,k)$, some direct computations yield that
\[ \overline{p}_{-k}\left(1 \mid \text {no } 1_{1}^{\prime} \mathrm{s}\right) \overline{p}_{-k}\left(1 \mid \text {no } 1_{2}^{\prime} \mathrm{s}\right)=(2k-1)^2,\]
\[ \overline{p}_{-k}\left(2 \mid \text {no } 1_{1}^{\prime} \mathrm{s} \text { and no } 1_{2}^{\prime} \mathrm{s}\right)=2k^2-2k+1,\]
so (\ref{ie12}) holds for $k \ge 2$.
Next, we assume $a\ge 2$. For $\lambda=\left(\lambda_{1}, \lambda_{2}, \cdots, \lambda_{t}\right) \in \overline{P}_{-k}(a+b),$ let
\[ i=i(\lambda)=\max \left\{j \in \mathbb{N} \mid 1 \leq j \leq t, \lambda_{j}+\cdots+\lambda_{t} \geq b\right\}. \]
Moreover, let $\lambda_{i}=x+y(x=x(\lambda), y=y(\lambda))$ such that
\[ x+\lambda_{i+1}+\cdots+\lambda_{t}=b \quad \text { and } \quad y+\lambda_{1}+\cdots+\lambda_{i-1}=a. \]
Note that $0<x \leq \lambda_{i}.$ Now we define a map
 \[f_{k}: \overline{P}_{-k}\left(a+b \mid \text {no } 1_{1}^{\prime} \mathrm{s} \text { and no } 1_{2}^{\prime} \mathrm{s}\right) \rightarrow \overline{P}_{-k}\left(a \mid \text {no } 1_{1}^{\prime} \mathrm{s}\right) \oplus \overline{P}_{-k}\left(b \mid \text {no } 1_{2}^{\prime} \mathrm{s}\right) \]
as follows. For $\lambda=\left(\lambda_{1}, \lambda_{2}, \cdots, \lambda_{t}\right) \in \overline{P}_{-k}\left(a+b \mid \text {no } 1_{1}^{\prime} \mathrm{s} \text { and no } 1_{2}^{\prime} \mathrm{s}\right),$
\begin{equation*}
f_{k}(\lambda):= \begin{cases}(\lambda_{1}, \lambda_{2}, \cdots, \lambda_{i-1} ; \lambda_{i}, \cdots, \lambda_{t}), & \text { if } y=0 ; \\
 (\lambda_{1}, \lambda_{2}, \cdots, \lambda_{i-1}, \overline{y}_{c\left(\lambda_{i}\right)}; \lambda_{i+1}, \cdots, \lambda_{t}, 1_{1}^{x}), & \text { if } y \geq 1, \lambda_{i} \text { is overlined }; \\
(\lambda_{1}, \lambda_{2}, \cdots, \lambda_{i-1}, y_{c\left(\lambda_{i}\right)} ; \lambda_{i+1}, \cdots, \lambda_{t}, 1_{1}^{x}), & \text { if } y \geq 2, \lambda_{i} \text { is non-overlined } ; \\
(\lambda_{1}, \lambda_{2}, \cdots, \lambda_{i-1}, 1_{c\left(\lambda_{i}\right)} ; \lambda_{i+1}, \cdots, \lambda_{t}, 1_{1}^{x}), & \text { if } y=1, c(\lambda_{i}) \neq 1, \lambda_{i} \text { is non-overlined } ; \\
(\lambda_{1}, \lambda_{2}, \cdots, \lambda_{i-2}, 1_{2}^{2}, \overline{1}_{1} ;  \\
   \lambda_{i+1}, \cdots, \lambda_{t}, 1_{1}^{x}), & \text { if } y=1, c\left(\lambda_{i}\right)=1, \lambda_{i-1}=2_{1},\\ &  \text { and } \lambda_{i} \text { is non-overlined } ;\\
 (\lambda_{1}, \lambda_{2}, \cdots, \lambda_{i-2}, (\lambda_{i-1}-1)_{c\left(\lambda_{i-1}\right)}, 1_{2}^{2} ;  \\
   \lambda_{i+1}, \cdots, \lambda_{t}, 1_{1}^{x}), & \text { if } y=1, c\left(\lambda_{i}\right)=1, \lambda_{i-1}\neq 2_{1},\\ &  \text { and } \lambda_{i}, \lambda_{i-1} \text { are non-overlined } ;\\
(\lambda_{1}, \lambda_{2}, \cdots, \lambda_{i-2}, (\overline{\lambda_{i-1}-1})_{c\left(\lambda_{i-1}\right)}, 1_{2}^{2} ;  \\
 \lambda_{i+1}, \cdots, \lambda_{t}, 1_{1}^{x}), & \text { if } y=1, c\left(\lambda_{i}\right)=1,  \lambda_{i}  \text { is non-overlined }, \\ &  \text { and } \lambda_{i-1}  \text { is overlined }.
 \end{cases}
\end{equation*}\par
Here, $c\left(\lambda_{i}\right)$ denotes the color of part $\lambda_{i}$. For the sake of convenience, in the cases $y\neq 0$ and $\lambda_{t}=\overline{1}_{1}$, we denote the partitions $(\lambda_{i+1}, \cdots, \lambda_{t-1}, 1_{1}^{x}, \lambda_{t})$ of $b$ by $(\lambda_{i+1}, \cdots, \lambda_{t-1}, \lambda_{t}, 1_{1}^{x})$ in the images of $f_{k}$, and 
if $y=1, c\left(\lambda_{i}\right)=1, \lambda_{i-1}=\overline{2}_{2}$ and $\lambda_{i}$ is non-overlined, we denote the partitions $(\lambda_{1}, \lambda_{2}, \cdots, \lambda_{i-2}, 1_{2}^{2}, \overline{1}_{2})$ of $a$ by $(\lambda_{1}, \lambda_{2}, \cdots, \lambda_{i-2}, \overline{1}_{2},1_{2}^{2})$. Clearly, $f_{k}$ is well-defined and since $a\ge 2$, and $y+\lambda_{1}+\cdots+\lambda_{i-1}=a$, then $\lambda_{i-1}$ exists when $y=1$. Furthermore, since we arrange the parts of each partition in a unique way and $x>0$, we have $\lambda_{i-1}\ge \lambda_{i}=x+y\ge 2$ for the case $y=1$, which implies $\lambda_{i-1} \neq \overline{2}_{1},$ so the images of the fifth and seventh cases are disjoint. And since $b\ge 1$, then $\lambda_{i-1}\neq \overline{1}_{1}$ for the case $y=0$, it follows that the pair $\left(\cdots, \overline{1}_{2}, \overline{1}_{1} ;\cdots \right)$ is not in the images of $f_{k}$, but $\left(\cdots, \overline{1}_{2}, \overline{1}_{1} ; \cdots \right) \in \overline{P}_{-k}\left(a \mid \text {no } 1_{1}^{\prime} \mathrm{s}\right) \oplus \overline{P}_{-k}\left(b \mid \text {no } 1_{2}^{\prime} \mathrm{s}\right)$. Therefore, $f_{k}$ is injective but not onto, and we are done.
\end{proof}
\begin{lemma}\label{le6}
If $a$ is a positive integer and $k \ge 2$, then
\begin{equation}\label{ie13}
\overline{p}_{-k}\left(a \mid \textup {no } 1_{1}^{\prime} \mathrm{s}\right) \overline{p}_{-k}(1)> \overline{p}_{-k}\left(a+1 \mid \textup {no } 1_{1}^{\prime} \mathrm{s}\right).
\end{equation}
\end{lemma}
\begin{proof}
Firstly, for $a=1$, we compute directly
\[ \overline{p}_{-k}\left(1 \mid \text {no } 1_{1}^{\prime} \mathrm{s}\right) \overline{p}_{-k}(1)=4k^2-2k,\]
\[ \overline{p}_{-k}\left(2 \mid \text {no } 1_{1}^{\prime} \mathrm{s}\right)=2k+k-1+k-2+ \cdots +1+(k-1)k+\binom{k}{2}=2k^2,\]
and clearly (\ref{ie13}) holds for $k\ge 2$. Now we assume $a\ge2.$
Given a partition $\lambda=\left(\lambda_{1}, \lambda_{2}, \cdots, \lambda_{t}\right) \in \overline{P}_{-k}\left(a+1 \mid \text {no } 1_{1}^{\prime} s\right),$ then $\lambda_{j}\neq 1_{1}(1\le j \le t)$. \par
Define a map
\[ g_{k}: \overline{P}_{-k}\left(a+1 \mid \text {no } 1_{1}^{\prime} \mathrm{s}\right) \rightarrow \overline{P}_{-k}\left(a \mid \text {no } 1_{1}^{\prime} \mathrm{s}\right) \oplus \overline{P}_{-k}(1) \]
by
\begin{equation*}
g_{k}(\lambda): = \begin{cases} (\lambda_{1}, \lambda_{2}, \cdots, \lambda_{t-1}, \left(\overline{\lambda_{t}-1}\right)_{c\left(\lambda_{t}\right)} ; 1_{1}), & \text { if } \lambda_{t} \geq 2, \lambda_{t} \text { is overlined }; \\
(\lambda_{1}, \lambda_{2}, \cdots, \lambda_{t-1}, \left(\lambda_{t}-1\right)_{c\left(\lambda_{t}\right)} ; 1_{1}), & \text { if } \lambda_{t} \geq 2, \lambda_{t}\neq 2_{1}, \lambda_{t} \text { is  non-overlined }; \\
(\lambda_{1}, \lambda_{2}, \cdots, \lambda_{t-2}, (\overline{\lambda_{t-1}-1})_{c\left(\lambda_{t-1}\right)}, 1_{2}^{2} ; 1_{1}), & \text { if } \lambda_{t}=2_{1},  \lambda_{t-1} \text { is  overlined };\\
(\lambda_{1}, \lambda_{2}, \cdots, \lambda_{t-2}, (\lambda_{t-1}-1)_{c\left(\lambda_{t-1}\right)}, 1_{2}^{2} ; 1_{1}), & \text { if } \lambda_{t}=2_{1}, \lambda_{t-1}\neq 2_{1},\\
& \text{ and } \lambda_{t-1} \text { is  non-overlined };\\
(\lambda_{1}, \lambda_{2}, \cdots, \lambda_{t-2}, 1_{2}^{2}, \overline{1}_{1} ; 1_{1}), & \text { if } \lambda_{t}=2_{1}, \lambda_{t-1}= 2_{1};\\
(\lambda_{1}, \lambda_{2}, \cdots, \lambda_{t-1}; \lambda_{t}), & \text { if } \lambda_{t}=1.
\end{cases}
\end{equation*}\par
Here, in the case of $\lambda_{t}=2_{1}, \lambda_{t-1}=\overline{2}_{2}$, we denote the partitions $(\lambda_{1}, \lambda_{2}, \cdots, \lambda_{t-2}, 1_{2}^{2}, \overline{1}_2)$ of $a$ as $(\lambda_{1}, \lambda_{2}, \cdots, \lambda_{t-2}, \overline{1}_2, 1_{2}^{2})$, and the last case $\lambda_{t}=1$ means that the size of $\lambda_{t}$ is 1, except for $\lambda_{t}=1_{1}$. Clearly, $g_{k}$ is well-defined. And if the size of  $\lambda_{t}$ is $2$, since $a\ge 2$, then $\lambda_{t-1}$ does not vanish and $\lambda_{t-1}\ge 2$. Furthermore, by the uniqueness of arrangement for parts, we have $\lambda_{t-1} \neq \overline{2}_{1}$ when
$\lambda_{t}=2_{1}$, so the images of any two cases are disjoint. And the pairs $(\cdots, \overline{1}_{2}, \overline{1}_{1} ; 1_{c})$ and $(\cdots, \overline{1}_{2}, \overline{1}_{1} ; \overline{1}_{c}) \in \overline{P}_{-k}\left(a \mid \text {no } 1_{1}^{\prime} \mathrm{s}\right) \oplus \overline{P}_{-k}(1)$, where $1\le c\le k$, but they are not the images of $g_{k}$. Hence, $g_{k}$ is one-to-one but not onto, which implies that the result follows.
\end{proof}
\begin{lemma}\label{le7}
If $a,b,k$ are positive integers with $a\ge b\ge 1$ and $k \ge 2$, then
\[\overline{p}_{-k}\left(a \mid \textup {no } 1_{1}^{\prime} \mathrm{s}\right) \overline{p}_{-k}(b)> \overline{p}_{-k}\left(a+b \mid \textup {no } 1_{1}^{\prime} \mathrm{s}\right).\]
\end{lemma}
\begin{proof}
For $b=1$, it has been proved in Lemma \ref{le6}. So we assume that $b \ge 2$. Let $n=a+b$, we will apply induction on $n$.\par
Base case($a=b=2,n=4$),
\[\overline{p}_{-k}\left(2 \mid \text {no } 1_{1}^{\prime} \mathrm{s}\right) \overline{p}_{-k}(2)=2k^2(2k^2+2k),\]
\[\overline{p}_{-k}\left(4 \mid \text {no } 1_{1}^{\prime} \mathrm{s}\right)=\frac{2k^4+8k^3+10k^2-2k}{3},\]
and
\begin{align}
\overline{p}_{-k}\left(2 \mid \text {no } 1_{1}^{\prime} \mathrm{s}\right) \overline{p}_{-k}(2)-\overline{p}_{-k}\left(4 \mid \text {no } 1_{1}^{\prime} \mathrm{s}\right)={} & \frac{10k^4+4k^3-10k^2+2k}{3}>0,\text { for } k\ge 2. \notag
\end{align}\par
Inductive step: suppose the inequality is true for smaller sum $a+b$. Then by the inductive hypothesis and Lemma \ref{le5} and Lemma \ref{le6}, we have
\begin{align}
 &\overline{p}_{-k}\left(a+b \mid \text {no } 1_{1}^{\prime} \mathrm{s}\right)\notag \\
 &= \overline{p}_{-k}\left(a+b \mid \text {no } 1_{1}^{\prime} \mathrm{s} \text { and at least one } 1_{2}^{\prime} \mathrm{s}\right)+ \overline{p}_{-k}\left(a+b \mid \text {no } 1_{1}^{\prime} \mathrm{s} \text { and no } 1_{2}^{\prime} \mathrm{s}\right) \notag \\
&= \overline{p}_{-k}\left(a+b-1 \mid \text {no } 1_{1}^{\prime} \mathrm{s}\right)+\overline{p}_{-k}\left(a+b \mid \text {no } 1_{1}^{\prime} \mathrm{s} \text { and no } 1_{2}^{\prime} \mathrm{s}\right) \notag \\
&< \overline{p}_{-k}\left(a \mid \text {no } 1_{1}^{\prime} \mathrm{s}\right) \overline{p}_{-k}(b-1)+\overline{p}_{-k}\left(a+b \mid \text {no } 1_{1}^{\prime} \mathrm{s} \text { and no } 1_{2}^{\prime} \mathrm{s}\right) \notag \\
&< \overline{p}_{-k}\left(a \mid \text {no } 1_{1}^{\prime} \mathrm{s}\right) \overline{p}_{-k}\left(b \mid \text {at least one } 1_{2}^{\prime} \mathrm{s}\right)+\overline{p}_{-k}\left(a \mid \text {no } 1_{1}^{\prime} \mathrm{s}\right) \overline{p}_{-k}\left(b \mid \text {no } 1_{2}^{\prime} \mathrm{s}\right) \notag \\
&=\overline{p}_{-k}\left(a \mid \text {no } 1_{1}^{\prime} s\right) \overline{p}_{-k}(b). \notag
\end{align}
This completes the proof.
\end{proof}

Now, we are ready to prove Theorem \ref{th5} by induction.
\begin{proof}[Proof of Theorem \ref{th5}]
Let $n=a+b$. We apply induction on $n$.\par
Base case: It can be verified that the inequality holds for $n=2(a=b=1)$ for any $k\ge 2$.\par
Inductive step: Assume that $n\ge 3$ and the inequality holds for $n-1$. Without loss of generality, suppose $a\ge b$, so $a\ge 2$. Thus, by inductive hypothesis,
\[\overline{p}_{-k}(a+b-1)<\overline{p}_{-k}(a-1) \overline{p}_{-k}(b).\]
According to Lemma \ref{le7}, we have
\[\overline{p}_{-k}\left(a \mid \text {no } 1_{1}^{\prime} \mathrm{s}\right) \overline{p}_{-k}(b)> \overline{p}_{-k}\left(a+b \mid \text {no } 1_{1}^{\prime} \mathrm{s}\right),\]
then
\begin{align}
\overline{p}_{-k}(a+b)= {} & \overline{p}_{-k}\left(a+b \mid \text {at least one } 1_{1}^{\prime} s\right)+\overline{p}_{-k}\left(a+b \mid \text {no } 1_{1}^{\prime} s\right) \notag \\
= {} & \overline{p}_{-k}(a+b-1)+\overline{p}_{-k}\left(a+b \mid \text {no } 1_{1}^{\prime} \mathrm{s}\right) \notag \\
< {} & \overline{p}_{-k}(a-1) \overline{p}_{-k}(b)+\overline{p}_{-k}\left(a \mid \text {no } 1_{1}^{\prime} \mathrm{s}\right) \overline{p}_{-k}(b) \notag \\
= {} & \overline{p}_{-k}(a-1) \overline{p}_{-k}(b)+\left\{\overline{p}_{-k}(a)-\overline{p}_{-k}\left(a \mid \text {at least one } 1_{1}^{\prime} \mathrm{s}\right)\right\} \overline{p}_{-k}(b) \notag \\
= {} & \overline{p}_{-k}(a-1) \overline{p}_{-k}(b)+\left\{\overline{p}_{-k}(a)-\overline{p}_{-k}(a-1)\right\} \overline{p}_{-k}(b) \notag \\
= {} & \overline{p}_{-k}(a) \overline{p}_{-k}(b). \notag
\end{align}
Therefore, by the principle of mathematical induction, the inequality holds for $n\ge 2$.
\end{proof}


\begin{thebibliography}{99}
\bibitem{Ala} Alanazi A A, Gagola S M, Munagi A O. Combinatorial proof of a partition inequality of Bessenrodt-Ono[J]. Annals of Combinatorics, 2017, 21(3): 331-337.

\bibitem{And} Andrews G E. The theory of partitions[M]. Cambridge university press, 1998.

\bibitem{Eri} Andrews G E, Eriksson K. Integer partitions[M]. Cambridge University Press, 2004.

\bibitem{Che} Chern S, Fu S, Tang D. Some inequalities for k-colored partition functions[J]. The Ramanujan Journal, 2018, 46(3): 713-725.

\bibitem{Co} Corteel S, Lovejoy J. Overpartitions[J]. Transactions of the American Mathematical Society, 2004, 356(4): 1623-1635.

\bibitem{En} Engel B. Log-concavity of the overpartition function[J]. The Ramanujan Journal, 2017, 43(2): 229-241.

\bibitem{Han} Han G N. The Nekrasov-Okounkov hook length formula: refinement, elementary proof, extension and applications[C]//Annales de l'Institut Fourier. 2010, 60(1): 1-29.

\bibitem{Neu} Heim B, Neuhauser M. The Dedekind eta function and D'Arcais-type polynomials[J]. Research in the Mathematical Sciences, 2020, 7(1): 1-8.

\bibitem{Hel} Heim B, Neuhauser M, Troger R. Polynomization of the Bessenrodt-Ono inequality[J]. Annals of Combinatorics, 2020, 24(4): 697-709.

\bibitem{Liu} Liu E Y S, Zhang H W J. Inequalities for the overpartition function[J]. The Ramanujan Journal, 2021, 54(3): 485-509.

\bibitem{Ne} Newman M. An identity for the coefficients of certain modular forms[J]. Journal of the London Mathematical Society, 1955, 1(4): 488-493.

\bibitem{Zu} Zuckerman H S. On the coefficients of certain modular forms belonging to subgroups of the modular group[J]. Transactions of the American Mathematical Society, 1939, 45(2): 298-321.

\end{thebibliography}
\end{document}